\documentclass[12pt,reqno]{amsart}
\usepackage{amssymb}
\usepackage{amsthm}
\usepackage{amsfonts}
\usepackage{amsmath,cite}
\usepackage[compact]{titlesec}
\usepackage{xcolor}
\usepackage{helvet}
\usepackage{enumitem}
\usepackage[a4paper,left=2.5cm,right=2.5cm,bottom=3cm,top=3.5cm]{geometry}

\titleformat{\section}{\raggedright\normalfont\bfseries}{\thesection.}{0.4em}{}
\titlespacing*{\section} {0pt}{3.5ex plus 1ex minus .2ex}{2.3ex plus .2ex}
\theoremstyle{definition}
\newtheorem{definition}{Definition}
\newtheorem{theorem}[definition]{Theorem}
\newtheorem{corollary}[definition]{Corollary}

\newtheorem{remark}[definition]{Remark}

\begin{document}

\thispagestyle{plain}

{\bf{International Journal of Mathematical Analysis and Applications}}

{\bf{2018; 5(2): 39-43}}

{\bf{http://www.aascit.org/journal/ijmaa}}

\title{Harmonic univalent functions defined by \\ $q$-calculus operators}

\author{Jay M. Jahangiri}
\address{Mathematical Sciences, Kent State University, 
Kent, Ohio, U.S.A.}
\email{jjahangi@kent.edu; http://www.kent.edu/math/profile/jay-jahangiri}

\let\thefootnote\relax

\maketitle

\begin{abstract}
{\small {The fractional q-calculus is the q-extension of the ordinary fractional calculus and dates back to early 20-th century. The theory of q-calculus operators are used in various areas of science such as ordinary fractional calculus, optimal control, q-difference and q-integral equations, and also in the geometric function theory of complex analysis. In this article, for the first time, we apply certain q-calculus operators to complex harmonic functions and obtain sharp coefficient bounds, distortion theorems and covering results.}}
\end{abstract}

\section{Introduction}

The theory of q-calculus operators are used in describing and solving various problems in applied science such as ordinary fractional calculus, optimal control, q-difference and q-integral equations, as well as geometric function theory of complex analysis. The fractional q-calculus is the q-extension of the ordinary fractional calculus and dates back to early 20-th century (e.g. see \cite{jackson} or \cite{andrews}). 
For $0<q<1$ and for positive integers $n$, the $q$-integer number $n$, denoted by $[n]_q$, is defined as
\begin{equation*}
[n]_q=\frac{1-q^n}{1-q}=\sum_{k=0}^{n-1}q^k.
\end{equation*}
Using the differential calculus, one can easily verify that
\begin{equation*}
\lim_{{q}\rightarrow{1^-}}[n]_q=n.
\end{equation*}
Let $\mathcal{A}$ denote the class of functions that are analytic in the open unit disc $\mathbb{D}:=\{z\in 
\mathbb{C}:~|\mathit{z}|<1\}$ and let $\mathcal{A^0}$ be the subclass of $\mathcal{A}$ consisting of functions $h$ with the normalization $h(0)=h^{\prime }(0)-1={0}$. The $q$-difference operator of $q$-calculus operated on the function $h$ (e.g. see \cite{ismail}, \cite{purohit}, \cite{kanas}, \cite{agrawal}, \cite{govin})  is defined by
\begin{equation}\label{q}
{\partial}_qh(z)=\frac{h(z)-h(qz)}{(1-q)z}
\end{equation}
where
\begin{equation*}
\lim_{{q}\rightarrow{1^-}}{\partial}_qh(z)=h'(z).
\end{equation*}

A successive application of the $q$-difference operator of $q$-calculus as defined in (\ref{q}) yields to what we call "Salagean $q$-differential operator." Therefore,
for functions $h\ {\in}\ \mathcal{A}$  of the form
\begin{equation}\label{h}
h(z)=z+\sum_{n=2}^{\infty}a_nz^n
\end{equation}
the Salagean $q$-differential operator of $h$, denoted by ${\mathcal{D}}_q^mh(z)$, is defined by
\begin{equation*}
{\mathcal{D}}_q^0h(z)=h(z),
\end{equation*}
\begin{equation*}
{\mathcal{D}}_q^1h(z)=z{\partial}_qh(z)=\frac{h(z)-h(qz)}{1-q},\ ...\ ,
\end{equation*}
\begin{equation}\label{D(q,m)}
\begin{split}
{\mathcal{D}}_q^mh(z)=&z{\partial}_q{{\mathcal{D}}_q^{m-1}h(z)} \\
=&h(z)*(z+\sum_{n=2}^{\infty}[n]_q^mz^n) \\
=&z+\sum_{n=2}^{\infty}[n]_q^ma_nz^n,
\end{split}
\end{equation}
where $m$ is a positive integer and the operator $*$ stands for the Hadamard product or convolution of two analytic power series.

The operator ${\mathcal{D}}_q^m$ is called Salagean $q$-differential operator because
\begin{equation*}
\lim_{{q}\rightarrow{1^-}}{\mathcal{D}}_q^mh(z)=z+\sum_{n=2}^{\infty}n^ma_nz^n
\end{equation*}
which is the famous Salagean operator \cite{salagean}.

It is the aim of this article to define the $q$-difference operator of $q$-calculus operated on the complex functions that are harmonic in $\mathbb{D}$ and obtain sharp coefficient bounds, distortion theorems and covering results. To the best of our knowledge, no such application has yet been published and the results presented here are new in their own kind. To this end, we consider the family of complex-valued harmonic functions $f=u+iv$ defined in $\mathbb{D}$, where $u$ and $v$ are real harmonic in $\mathbb{D}$. Such functions can be expressed as $f=h+\overline{g}$, where ${h}\in{\mathcal{A^0}}$ is given by (\ref{h}) and ${g}\in{\mathcal{A}}$ has the following power series expansion
\begin{equation}\label{g(z)}
g(z)=\sum_{n=1}^{\infty}b_nz^n;\ \ |b_1|<1.
\end{equation}
Clunie and Sheil-Small in their remarkable paper \cite{C-S-S} explored the functions of the form $f=h+\overline{g}$ that are locally one-to-one, sense-preserving and harmonic in $\mathbb{D}$.  By Lewy's Theorem (see \cite{Lewy} or \cite{C-S-S}), a necessary and sufficient condition for the harmonic function $f=h+\overline{g}$ to be locally one-to-one and orientation-preserving in $\mathbb{D}$ is that its Jacobian $J_f=|h^{\prime}|^2-|g^{\prime}|^2$ is positive or equivalently, if and only if $h^{\prime}(z){\ne}0$ in $\mathbb{D}$ and the second complex dilatation $\omega$ of $f$ satisfies $|\omega|=|g^{\prime}/h^{\prime}|<1$ in $\mathbb{D}$.

We define the Salagean $q$-differential operator of harmonic functions $f=h+\overline{g}$ by
\begin{equation}\label{D(q)}
{\mathcal{D}}_q^mf(z)={\mathcal{D}}_q^mh(z)+(-1)^m\overline{{\mathcal{D}}_q^mg(z)}
\end{equation}
where ${\mathcal{D}}_q^m$ is defined by (\ref{D(q,m)}) and $h$ and $g$ are of the form (\ref{h}) and (\ref{g(z)}), respectively.

For $0\ {\le}\ \alpha<1$ let $\mathcal{H}_q^m(\alpha)$ denote the family of harmonic functions $f=h+\overline{g}$ so that
\begin{equation}\label{H(q)}
{\Re}\left(\frac{{\mathcal{D}}_q^{m+1}f(z)}{{\mathcal{D}}_q^mf(z)}\right)
 {\ge}\alpha,
\end{equation}
where ${\mathcal{D}}_q^mf(z)$ is defined by (\ref{D(q)}).

We further denote by $\mathcal{\overline{H}}_q^m(\alpha)$ the subclass of $\mathcal{H}_q^m(\alpha)$ consisting of harmonic functions $f=h+\overline{g}$ so that $h$ and $g$ are of the form
\begin{equation*}
h(z)=z-\sum_{n=2}^{\infty}a_nz^n\ ,\ \ g(z)=\sum_{n=1}^{\infty}b_nz^n\ ,\ \ \ a_n\ {\ge}\ 0,\ b_n\ {\ge}\ 0.
\end{equation*}

\section{Main Results}
In the following theorem we shall determine coefficient bounds for harmonic functions in $\mathcal{H}_q^m(\alpha)$ and $\mathcal{\overline{H}}_q^m(\alpha)$.
\begin{theorem} \label{Theorem Main}
For $0\ {\le}\ \alpha<1$ and $f=h+\overline{g}$ let
\begin{equation}\label{th1}
\sum_{n=2}^{\infty}[n]_q^m([n]_q-\alpha)|a_n|+\sum_{n=1}^{\infty}[n]_q^m([n]_q+\alpha)|b_n|\ {\le}\ 1-\alpha
\end{equation}
where $h$ and $g$ are, respectively, given by (\ref{h}) and (\ref{g(z)}). Then

\noindent i) $f$ is harmonic univalent in $\mathbb{D}$ and ${f}\in{\mathcal{H}_q^m(\alpha)}$ if the inequality (\ref{th1}) holds.

\noindent ii) $f$ is harmonic univalent in $\mathbb{D}$ and ${f}\in{\mathcal{\overline{H}}_q^m(\alpha)}$ if and only if the inequality (\ref{th1}) holds.
\end{theorem}
The equality in (\ref{th1}) occurs for harmonic functions
\begin{equation*}
f(z)=z+\sum_{n=2}^{\infty}\frac{1-\alpha}{[n]_q^m([n]_q-\alpha)}x_nz^n+\overline{\sum_{n=1}^{\infty}\frac{1-\alpha}{[n]_q^m([n]_q+\alpha)}y_nz^n}
\end{equation*}
where $\sum_{n=2}^{\infty}|x_n|+\sum_{n=1}^{\infty}|y_n|=1$.

As special cases of Theorem 1, we obtain the following two corollaries.

\begin{corollary}
For $m=0$, Theorem 1 yields the results obtained by the author in (\cite{jay1}, Theorems 1 and 2). This can be easily verified since
\begin{equation*}
{\mathcal{D}}_q^0f(z)={\mathcal{D}}_q^0h(z)+(-1)^0\overline{{\mathcal{D}}_q^0g(z)}=h(z)+\overline{g(z)}.
\end{equation*}
\end{corollary}
\begin{corollary}
For $q{\rightarrow}1^-$, Theorem 1 yields the results obtained in (\cite{jay2}, Theorems 1 and 2)  since 
\begin{eqnarray*}
\lim_{{q}\rightarrow{1^-}}{\mathcal{D}}_q^mf(z)&=&\lim_{{q}\rightarrow{1^-}}\left\{{\mathcal{D}}_q^mh(z)+(-1)^m\overline{{\mathcal{D}}_q^mg(z)}\right\} \\
&=&z+\sum_{n=2}^{\infty}n^ma_nz^n+(-1)^m\overline{\sum_{n=1}^{\infty}n^mb_nz^n}.
\end{eqnarray*}
\end{corollary}

In the next theorem we determine the extreme points of the closed convex hull of ${\mathcal{\overline{H}}_q^m(\alpha)}$ denoted by $clco{{\mathcal{\overline{H}}_q^m(\alpha)}}$.
\begin{theorem}\label{convexhull}
$f\ {\in}\ clco{{\mathcal{\overline{H}}_q^m(\alpha)}}$ if and only if
\begin{equation}\label{th4}
f(z)=\sum_{n=1}^{\infty}\left(X_nh_n(z)+\overline{Y_ng_n(z)}\right),
\end{equation}
where $h_1(z)=z$, $h_n(z)=z-\frac{1-\alpha}{[n]_q^m([n]_q-\alpha)}z^n$; $(n=2,3,...)$, $g_n(z)=z+(-1)^m\frac{1-\alpha}{[n]_q^m([n]_q+\alpha)}z^n$; $(n=1,2,3...)$, $\sum_{n=1}^{\infty}(X_n+Y_n)=1$, $X_n\ {\ge}\ 0$, and $Y_n\ {\ge}\ 0$. In particular, the extreme points of ${\mathcal{\overline{H}}_q^m(\alpha)}$ are $\{h_n\}$ and $\{g_n\}$.
\end{theorem}
Finally, we give the distortion bounds for functions in ${\mathcal{\overline{H}}_q^m(\alpha)}$ which yield a covering result for the class ${\mathcal{\overline{H}}_q^m(\alpha)}$.
\begin{theorem}\label{distortion}
If $f\ {\in}\ {\mathcal{\overline{H}}_q^m(\alpha)}$, then for $|z|=r<1$ we have the distortion bounds
\begin{equation*}
(1-b_1)r-\frac{1}{[2]_q^m}\left(\frac{1-\alpha}{[2]_q-\alpha}-\frac{1+\alpha}{[2]_q-\alpha}b_1\right)r^2{\le}|f(z)|{\le}
(1+b_1)r+\frac{1}{[2]_q^m}\left(\frac{1-\alpha}{[2]_q-\alpha}-\frac{1+\alpha}{[2]_q-\alpha}b_1\right)r^2.
\end{equation*}
\end{theorem}
As a consequence of Theorem \ref{distortion}, we obtain the following
\begin{corollary}\label{covering}
If $f\ {\in}\ {\mathcal{\overline{H}}_q^m(\alpha)}$, then
\begin{equation*}
\left\{{\omega}:\ |{\omega}|\ <\ \frac{[2]_q^{m+1}-1-([2]_q^m-1)\alpha}{[2]_q^m([2]_q-\alpha)}\left(1-\frac{[2]_q-\alpha}{[2]_q+\alpha}b_1 \right)\right\}\ {\subset}\ f(\mathbb{D}).
\end{equation*}
\end{corollary}
\begin{remark}
The above Theorems \ref{convexhull} and \ref{distortion} and Corollary \ref{covering} for $m=0$ yield the results obtained by the author in \cite{jay1} and for $q{\rightarrow}1^-$ yield the results in \cite{jay2}.
\end{remark}

\section{Proofs}

\begin{proof}[Proof of Theorem \ref{Theorem Main}]
\

Proof of Part (i): First we need to show that $f=h+\overline{g}$ is locally univalent and orientation-preserving in $\mathbb{D}$. It suffices to show that the second complex dilatation $\omega$ of $f$ satisfies $|\omega|=|g^{\prime}/h^{\prime}|<1$ in $\mathbb{D}$. This is the case since for $z=re^{i\theta}\ {\in}\ \mathbb{D}$ we have
\begin{eqnarray*}
|h^{\prime}(z)|&{\ge}&1-\sum_{n=2}^{\infty}n|a_n|r^{n-1}>1-\sum_{n=2}^{\infty}n|a_n|\ {\ge}\ 1-\sum_{n=2}^{\infty}\frac{[n]_q^m([n]_q-\alpha)}{1-\alpha}|a_n| \\
&{\ge}&\sum_{n=1}^{\infty}\frac{[n]_q^m([n]_q+\alpha)}{1-\alpha}|b_n|\ {\ge}\ \sum_{n=1}^{\infty}n|b_n|\ {\ge}\  \sum_{n=1}^{\infty}n|b_n|r^{n-1}\ {\ge}\ |g^{\prime}(z)|.
\end{eqnarray*}
To show that $f=h+\overline{g}$ is univalent in $\mathbb{D}$ we use an argument that is due to author (\cite{jay1}, Proof of Theorem 1). Suppose $z_1$ and $z_2$ are in $\mathbb{D}$ so that $z_1{\ne}z_2$. Since $\mathbb{D}$ is simply connected and convex, we have $z(t)=(1-t)z_1+tz_2\ {\in}\ \mathbb{D}$ for $0\ {\le}\ t\ {\le}\ 1$. Then we can write
\begin{equation*}
f(z_2)-f(z_1)=\int_0^1\left[(z_2-z_1)h^{\prime}(z(t))+\overline{(z_2-z_1)g^{\prime}(z(t))}\right]dt.
\end{equation*}
Dividing by $z_2-z_1{\ne}0$ and taking the real parts yield
\begin{eqnarray*}
\Re\frac{f(z_2)-f(z_1)}{z_2-z_1}&=&\int_0^1\Re{\left[h^{\prime}(z(t))+\frac{\overline{z_2-z_1}}{z_2-z_1}\overline{g^{\prime}(z(t))}\right]}dt \\
&>&\int_0^1{\left[\Re{h^{\prime}(z(t))}-|g^{\prime}(z(t))|\right]}dt.
\end{eqnarray*}
On the other hand, we observe that
\begin{eqnarray*}
\Re{h^{\prime}(z)}-|g^{\prime}(z)|&\ge&\Re{h^{\prime}(z)}-\sum_{n=1}^{\infty}n|b_n|{\ge}1-\sum_{n=2}^{\infty}n|a_n|-\sum_{n=1}^{\infty}n|b_n| \\
&\ge&1-\sum_{n=2}^{\infty}\frac{[n]_q^m([n]_q-\alpha)}{1-\alpha}|a_n|-\sum_{n=1}^{\infty}\frac{[n]_q^m([n]_q+\alpha)}{1-\alpha}|b_n|\ {\ge}\ 0.
\end{eqnarray*}
Therefore, $f=h+\overline{g}$ is univalent in $\mathbb{D}$.
It remains to show that the inequality (\ref{H(q)}) holds if the coefficients of the univalent harmonic function  $f=h+\overline{g}$ satisfy the condition (\ref{th1}). In other words, for $0\ {\le}\ \alpha\ <\ 1,$ we need to show that
\begin{equation*}
{\Re}\left(\frac{{\mathcal{D}}_q^{m+1}f(z)}{{\mathcal{D}}_q^mf(z)}\right)\ =\ {\Re}\left(\frac{{\mathcal{D}}_q^{m+1}h(z)+(-1)^{m+1}\overline{{\mathcal{D}}_q^{m+1}g(z)}}{{\mathcal{D}}_q^mh(z)+(-1)^m\overline{{\mathcal{D}}_q^mg(z)}}  \right)
 {\ge}\ \alpha.
\end{equation*}
Using the fact that ${\Re}({\omega})\ge{\alpha}$ if and only if $|1-{\alpha}+\omega|{\ge}|1+{\alpha}-\omega|,$ it suffices to show that
\begin{equation}\label{eq1}
|{\mathcal{D}}_q^{m+1}f(z)+(1-\alpha){\mathcal{D}}_q^{m}f(z)|-|{\mathcal{D}}_q^{m+1}f(z)-(1+\alpha){\mathcal{D}}_q^{m}f(z)|\ {\ge}\ 0.
\end{equation}
Substituting for
\begin{equation*}
{\mathcal{D}}_q^{m}f(z)=z+\sum_{n=2}^{\infty}[n]_q^ma_nz^n+(-1)^m\sum_{n=1}^{\infty}[n]_q^m\overline{b_n}{\overline{z}}^n
\end{equation*}
and
\begin{equation*}
{\mathcal{D}}_q^{m+1}f(z)=z+\sum_{n=2}^{\infty}[n]_q^{m+1}a_nz^n+(-1)^{m+1}\sum_{n=1}^{\infty}[n]_q^{m+1}\overline{b_n}{\overline{z}}^n
\end{equation*}
in the left hand side of the inequality (\ref{eq1}) we obtain
\begin{eqnarray*}
|{\mathcal{D}}_q^{m+1}f(z)&+&(1-\alpha){\mathcal{D}}_q^{m}f(z)|-|{\mathcal{D}}_q^{m+1}f(z)-(1+\alpha){\mathcal{D}}_q^{m}f(z)| \\
&{\ge}&2(1-\alpha)|z|\left\{ 1-\sum_{n=2}^{\infty}\frac{[n]_q^m([n]_q-\alpha)}{1-\alpha}|a_n||z|^{n-1}-\sum_{n=1}^{\infty}\frac{[n]_q^m([n]_q+\alpha)}{1-\alpha}|b_n||z|^{n-1}    \right\} \\
&{\ge}&2(1-\alpha)|z|\left\{ 1-\sum_{n=2}^{\infty}\frac{[n]_q^m([n]_q-\alpha)}{1-\alpha}|a_n|-\sum_{n=1}^{\infty}\frac{[n]_q^m([n]_q+\alpha)}{1-\alpha}|b_n| \right\}.
\end{eqnarray*}
This last expression is non-negative by (\ref{th1}), and so the proof is complete.

Proof of Part (ii): Since $\mathcal{\overline{H}}_q^m(\alpha)\ {\subset}\ \mathcal{H}_q^m(\alpha)$, we only need to prove the "only if" part of the theorem.
Let $f\ {\in}\ \mathcal{\overline{H}}_q^m(\alpha)$. Then by the required condition (\ref{H(q)}) we must have
\begin{equation*}
\Re\left\{\frac{(1-\alpha)z-\sum_{n=2}^{\infty}[n]_q^m([n]_q-\alpha)a_nz^n-(-1)^{2m}\sum_{n=1}^{\infty}[n]_q^m([n]_q+\alpha)b_n{\overline{z}}^n}{z-\sum_{n=2}^{\infty}[n]_q^ma_nz^n+(-1)^{2m}\sum_{n=1}^{\infty}[n]_q^mb_n{\overline{z}}^n}\right\}\ {\ge}\ 0.
\end{equation*}
This must hold for all values of $z$ in $\mathbb{D}$. So, upon choosing the values of $z$ on the positive real axis where $0\ \le\ z=r\ <\ 1$, we must have
\begin{equation}\label{eq2}
\frac{1-\alpha-\sum_{n=2}^{\infty}[n]_q^m([n]_q-\alpha)a_nr^{n-1}-\sum_{n=1}^{\infty}[n]_q^m([n]_q+\alpha)b_nr^{n-1}}{1-\sum_{n=2}^{\infty}[n]_q^ma_nr^{n-1}+\sum_{n=1}^{\infty}[n]_q^mb_nr^{n-1}}\ {\ge}\ 0.
\end{equation}
If the condition (\ref{th1}) does not hold, then the numerator in (\ref{eq2}) is negartive for $r$ sufficiently close to 1. Hence there exists $z_0=r_0$ in $(0,1)$ for which the left hand side of the inequality (10) is negative. This contradicts the required condition that $f\ {\in}\ \mathcal{\overline{H}}_q^m(\alpha)$ and so the proof is complete.
\end{proof}
\begin{proof}[Proof of Theorem \ref{convexhull}] For the functions of the form (\ref{th4}) we have
\begin{equation*}
f(z)=\sum_{n=1}^{\infty}(X_n+Y_n)z-\sum_{n=2}^{\infty}\frac{1-\alpha}{[n]_q^m([n]_q-\alpha)}X_nz^n+(-1)^m\sum_{n=1}^{\infty}\frac{1-\alpha}{[n]_q^m([n]_q+\alpha)}X_n{\overline{z}}^n.
\end{equation*}
This yields
\begin{equation*}
\sum_{n=2}^{\infty}\frac{[n]_q^m([n]_q-\alpha)}{1-\alpha}a_n+\sum_{n=1}^{\infty}\frac{[n]_q^m([n]_q+\alpha)}{1-\alpha}b_n=\sum_{n=2}^{\infty}X_n+\sum_{n=1}^{\infty}Y_n=1-X_1\ {\le}\ 1
\end{equation*}
and so $f\ {\in}\ clco{{\mathcal{\overline{H}}_q^m(\alpha)}}$. Conversely, let $f\ {\in}\ clco{{\mathcal{\overline{H}}_q^m(\alpha)}}$. Then by setting
\begin{equation*}
X_n=\frac{[n]_q^m([n]_q-\alpha)}{1-\alpha}a_n;\ (n=2,3,...),\ \ Y_n=\frac{[n]_q^m([n]_q+\alpha)}{1-\alpha}b_n;\ (n=1,2,3,...)
\end{equation*}
where $\sum_{n=1}^{\infty}(X_n+Y_n)=1$ we obtain the functions of the form (\ref{th4}) as required.
\end{proof}

\begin{proof}[Proof of Theorem \ref{distortion}] We shall only prove the right hand inequality in Theorem \ref{distortion}. The proof for the left hand inequality is similar and will be omitted. Let
$f\ {\in}\ {{\mathcal{\overline{H}}_q^m(\alpha)}}$. Taking the absolute value of $f$ we obtain
\begin{eqnarray*}
|f(z)|&{\le}&(1+b_1)r+\sum_{n=2}^{\infty}(a_n+b_n)r^n \\
&{\le}&(1+b_1)r+\sum_{n=2}^{\infty}(a_n+b_n)r^2 \\
&{\le}&(1+b_1)r+\frac{1-\alpha}{[2]_q^m([2]_q-\alpha)}\sum_{n=2}^{\infty}\left(\frac{[2]_q^m([2]_q-\alpha)}{1-\alpha}a_n +\frac{[2]_q^m([2]_q-\alpha)}{1-\alpha}b_n   \right)r^2 \\
&{\le}&(1+b_1)r+\frac{1-\alpha}{[2]_q^m([2]_q-\alpha)}\sum_{n=2}^{\infty}\left(\frac{[2]_q^m([n]_q-\alpha)}{1-\alpha}a_n +\frac{[2]_q^m([n]_q+\alpha)}{1-\alpha}b_n   \right)r^2 \\
&{\le}&(1+b_1)r+\frac{1-\alpha}{[2]_q^m([2]_q-\alpha)}\sum_{n=2}^{\infty}\left( 1-\frac{1+\alpha}{1-\alpha}b_1 \right)r^2 \\
&{\le}&(1+b_1)r+\frac{1}{[2]_q^m}\sum_{n=2}^{\infty}\left( \frac{1-\alpha}{[2]_q-\alpha}-\frac{1+\alpha}{[2]_q-\alpha}b_1 \right)r^2.
\end{eqnarray*}
\end{proof}

{\bf{Conclusion :}} The theory of q-calculus has been applied to many areas of mathematics and physics such as fractional calculus and quantum physics. But research on q-calculus in connection with function theory and especially harmonic univalent functions is fairly new and not much is published on this topic. Finding sharp coefficient bounds for harmonic univalent functions defined by q-calculus operators is of particular importance since any information can throw light on the study of the geometric properties of such functions. To date, not much is known about the coefficients of the entire class of harmonic univalent functions. Using a technique due to the author published in the Journal of Mathematical Analysis and Applications (1999), sharp coefficient bounds and related extremal functions for certain classes of harmonic univalent functions defined by q-calculus operators are determined. It is hoped that this can inspire further research by other investigators on this topic. \

{\bf{Acknowledgment :}} In memory of Maxwell O. Reade of University of Michigan, Ann Arbor; April 11, 1916 - April 13, 2016.

\end{document}